\documentclass{amsart}

\numberwithin{equation}{section}

\newtheorem{theorem}{Theorem}[section]
\newtheorem{lemma}[theorem]{Lemma}
\newtheorem{prop}[theorem]{Proposition}
\theoremstyle{definition}
\newtheorem{remark}[theorem]{Remark}

\usepackage{amsmath, amsthm}
\usepackage{textcomp}
\usepackage{amscd}
\usepackage{amsfonts}
\usepackage{amssymb}
\usepackage{tikz-cd}
\usepackage{mathabx}
\usepackage{mathtools}
\usepackage{graphicx}
\usepackage{hyperref}
\usepackage{epigraph}
\usepackage{underscore}
\usepackage{color}
\usepackage{mathdots}
\usepackage[utf8]{inputenc}

\DeclareMathOperator{\tr}{tr}

\newcommand{\nc}{\newcommand}
\nc{\on}{\operatorname}
\nc{\ul}{\underline}
\nc{\Gr}{{\on{Gr}}}
\nc{\CM}{{\mathcal M}}
\nc{\CN}{{\mathcal N}}
\nc{\CS}{{\mathcal S}}
\nc{\bN}{{\mathbf N}}
\nc{\BA}{{\mathbb A}}
\nc{\BC}{{\mathbb C}}
\nc{\BN}{{\mathbb N}}
\nc{\BP}{{\mathbb P}}
\nc{\BS}{{\mathbb S}}
\nc{\BW}{{\mathbb W}}
\nc{\fsp}{{\mathfrak{sp}}}
\nc{\fgl}{{\mathfrak{gl}}}
\nc{\fh}{{\mathfrak{h}}}
\nc{\fS}{{\mathfrak{S}}}
\nc{\se}{{\mathsf e}}
\nc{\ssf}{{\mathsf f}}
\nc{\sh}{{\mathsf h}}
\nc{\sx}{{\mathsf x}}
\nc{\sy}{{\mathsf y}}
\nc{\sw}{{\mathsf w}}

\usepackage{tikz}
\usetikzlibrary{external}
\usepackage{nicematrix}

\NiceMatrixOptions{transparent}

\usepackage[left=2cm,right=2cm,
    top=2cm,bottom=2cm,bindingoffset=0cm]{geometry}

\begin{document}

\title{Coulomb branch of a multiloop quiver gauge theory}

\author{Michael Finkelberg}
\address{M.F.:
  National Research University Higher School of Economics, Russian Federation,
  Department of Mathematics, 6 Usacheva st., Moscow 119048;
  Skolkovo Institute of Science and Technology;
  Institute for Information Transmission Problems}
 \email{fnklberg@gmail.com}
\author{Evgeny Goncharov}
\address{E.G.:
  National Research University Higher School of Economics, Russian Federation,
  Department of Mathematics, 6 Usacheva st, Moscow;
  Cambridge University, Center for Mathematical Sciences, Wilberforce Road, Cambridge, UK}
\email{eagoncharov@edu.hse.ru, eg555@cam.ac.uk}

\begin{abstract}
  We compute the Coulomb branch of a multiloop quiver gauge theory for the quiver with a single
  vertex, $r$ loops, one-dimensional framing, and $\dim V=2$. We identify it with a Slodowy
  slice in the nilpotent cone of the symplectic Lie algebra of rank $r$.
  Hence it possesses a symplectic resolution
  with $2r$ fixed points with respect to a Hamiltonian torus action. We also idenfity its
  flavor deformation with a base change of the full Slodowy slice.
\end{abstract}

\maketitle

\section{Introduction}

\subsection{}
We consider a multiloop analogue of the Jordan quiver having $r$ loops at a single vertex
instead of one. The corresponding quiver variety (the Higgs branch $\CM_H(r,2,1)$ of
the corresponding
quiver gauge theory) in the simplest case when the framing $W$ is 1-dimensional,
and $V$ is 2-dimensional, is an affine conical Poisson variety of dimension $8r-4$.
It possesses a symplectic resolution (a GIT quotient) $\widetilde\CM_H(r,2,1)\to\CM_H(r,2,1)$.
Contrary to the well-studied case of Nakajima quiver varieties for quivers without loop-edges,
when the symplectic resolutions are semismall, the resolution
$\widetilde\CM_H(r,2,1)\to\CM_H(r,2,1)$ is {\em small}: the central fiber is isomorphic
to $\BP^{2r-1}$. In particular, it has $2r$ fixed points with respect to a natural Hamiltonian
$\BC^\times$-action.

According to general expectations, it means that the corresponding Coulomb branch
$\CM_C(r,2,1)$ should also possess a symplectic resolution
$\widetilde\CM_C(r,2,1)\to\CM_C(r,2,1)$ with $2r$ fixed points with respect to a natural
Hamiltonian $\BC^\times$-action. In this note we prove that this is indeed true. Namely,
we identify $\CM_C(r,2,1)$ with the Slodowy slice $\CS(2r-2,1,1)$ in the nilpotent cone
$\CN_{\fsp(2r)}$ of the Lie algebra $\fsp(2r)$ to the adjoint orbit of a nilpotent element $e$
of Jordan type $(2r-2,1,1)$. Hence $\widetilde\CM_C(r,2,1)\to\CM_C(r,2,1)$ is identified with
the Springer resolution $\widetilde\CS(2r-2,1,1)\to\CS(2r-2,1,1)$. The centralizer
$Z_{Sp(2r)}(e)$ contains a subgroup $SL(2)$, and its Cartan torus acts on the Springer fiber
with $2r$ fixed points. Furthermore, we consider the flavor deformation $\ul\CM{}_C(r,2,1)$
of $\CM_C(r,2,1)$ over $\fh=\on{Lie}T_F=\BC z_1\oplus\ldots\oplus\BC z_r$ for the maximal torus
flavor group. We identify $\ul\CM{}_C(r,2,1)$ with a ramified cover $\ul\BS(2r-2,1,1)$ of the full
Slodowy slice $\BS(2r-2,1,1)\cong\BA^{r+4}$ obtained by identifying the spectra
of matrices in $\BS(2r-2,1,1)$ with $\{\pm z_1,\ldots,\pm z_r\}$ (so that $\fh$ is identified
with a Cartan Lie subalgebra of $\fsp(2r)$), see~Proposition~\ref{flavor}.
It seems likely that the quantized flavor deformation of the Coulomb branch is isomorphic to
the central extension of scalars (from $\BC[\fh/\BW]$ to $\BC[\fh]$) of the finite $\mathcal W$-algebra
$U(\fsp(2r),e)$.

Our proof is purely computational: we identify on the nose the equations defining $\CM(r,2,1)$
and $\CS(2r-2,1,1)$, cf.~\cite{y} for a general classification result underlying this identification.
Note that the symplectic resolution $\widetilde\CM_H(r,2,m)$ for
arbitrary dimension of the framing $W=\BC^m$ also has finitely many $\BC^\times$-fixed
points, and this is still true for $\widetilde\CM_H(r,3,m)$ if we increase the dimension
of $V$ to 3 (I.~Losev, private communication). Hence the corresponding Coulomb branches
are also expected to possess symplectic resolutions $\widetilde\CM_C(r,2,m),\widetilde\CM_C(r,3,m)$
with finitely many
$\BC^\times$-fixed points. However, we were unable to find such resolutions, see some
computations in~Section~\ref{3d}.

\subsection{Acknowledgments}
We are grateful to Ivan Losev for posing the problem and for the useful discussions.
We are also obliged to Travis Schedler for bringing~\cite{hm} to our attention.
%adviser Michael Finkelberg for plenty of fruitful conversations, his support and encouragement.
E.G.~is thankful to Ivan Chistyakov for help with the software.
M.F.\ was partially funded within the framework of the HSE University Basic Research Program
    and the Russian Academic Excellence Project `5-100'.

\section{Coulomb branch and Slodowy slice}

\subsection{Notations and the main result}
\label{notations}
We consider the representation space $\bN$ of the quiver with a single vertex,
$r$ loops, $V=\BC^2$, and the framing $W=\BC^1$. In other words, the gauge group $G=GL(V)=GL(2)$,
and $\bN=\fgl(V)^{\oplus r}\oplus V$. We denote the corresponding Coulomb branch~\cite{main} by
$\CM_C(r,2,1)$. It is an affine conical 4-dimensional Poisson variety equipped with
a Hamiltonian action of $SL(2)$, to be defined below.

We also consider a nilpotent element $e\in\fsp(2r)$ of Jordan type $(2r-2,1,1)$.
Its adjoint orbit has codimension 4 in the nilpotent cone $\CN_{\fsp(2r)}$
(see e.g.~\cite{KP}). We denote the corresponding Slodowy slice~\cite{Slodowy} in $\CN_{\fsp(2r)}$ by
$\CS(2r-2,1,1)$. It is an affine conical 4-dimensional Poisson variety. Since the
centralizer of $e$ in $Sp(2r)$ contains a subgroup $SL(2)$, the slice $\CS(2r-2,1,1)$
is equipped with a Hamiltonian action of $SL(2)$.

\begin{theorem}
  \label{isomo}
  There exists an $SL(2)$-equivariant Poisson isomorphism
  \[\Phi\colon\CM_C(r,2,1)\xrightarrow{\sim}\CS(2r-2,1,1).\]
  \end{theorem}

The proof is given in Section~\ref{coda} after a neccesary preparation.

\subsection{Equations for the Coulomb branch}
\label{coulomb}
We use the notations of~\cite[Appendix A]{quiver}, especially~\cite[A(iii),~(A.7)]{quiver}.
Thus we have rational \'etale coordinates $u_1,u_2,w_1,w_2$ on $\CM_C(r,2,1)$, and the regular
functions
\[E_1[1]= (w_1 - w_2)^{r-1} (u_1 - (-1)^r u_2),\ E_1[w] = (w_1 - w_2)^{r-1} (w_1u_1 - (-1)^r w_2u_2),\
E_2[1] = u_1 u_2,\]
\[F_1[1] = (w_1 - w_2)^{r-1} (w_1u_1^{-1} - (-1)^r w_2 u_2^{-1}),\
F_1[w] = (w_1 - w_2)^{r-1} (w_1^2 u_1^{-1} - (-1)^r w_2^2 u_2^{-1}),\ F_2[1] = w_1 w_2 u_1^{-1}u_2 ^{-1}.\]
These functions along with regular functions $w_1 + w_2,\ w_1w_2$ generate the Coulomb branch ring,
see~\cite[Proposition~3.1]{we}.
In fact, these 8 generators are redundant since \[w_1 w_2 = E_2[1]F_2[1],\
E_1[w] = (w_1 + w_2)E_1 [1]- (-1)^rE_2[1]F_1[1],\ F_1[w] = (w_1 + w_2)F_1[1] - (-1)^rF_2[1]E_1[1].\]
Thus we have 5 generators \[x_1:=E_1[1],\ x_2:=(-1)^rE_2[1],\ y_1:=F_1[1],\ y_2:=(-1)^rF_2[1],\
w:=w_1 + w_2.\]
The following lemma goes back to~\cite[7.2]{hm}.

\begin{lemma}
  \label{relatio}
  We have
  \begin{equation}
    \label{starlet}
    \BC[\CM_C(r,2,1)]=\BC[x_1,x_2,y_1,y_2,w]/\big((w^2-4x_2y_2)^r-(x_1^2y_2+x_2y_1^2+wx_1y_1)\big).
  \end{equation}
\end{lemma}

\begin{proof}
  The verification of the relation~(\ref{starlet}) is straightforward.
  It remains to check that the resulting surjective
  homomorphism from the RHS to the LHS is injective. If there were some additional relations,
  they would imply $\dim\CM_C(r,2,1)\leq 3$ which is nonsense.

  This finishes the proof, but is also instructive to compare the Hilbert series of the LHS
  and RHS in the natural grading where $\deg x_1=\deg y_1=2r-1,\ \deg x_2=\deg y_2=\deg w=2$.
  The Hilbert series of the RHS clearly is
  \[H(t) = \frac{1 - t^{4r}}{(1-t^2)^3(1-t^{2r-1})^2}.\]
  To compute the Hilbert series of the LHS we use the monopole formula~\cite[2(iii)]{main}.
  The dominant coweights of $GL(2)$ are the pairs of integers $\lambda = (n_1 \geq n_2)$.
  In notations of~\cite[2(iii)]{main} we have $P_G(t; \lambda) = (1 - t^2)^{-2}$ if $n_1 > n_2$
  and $P_G(t; \lambda) = (1 - t^2)^{-1}(1 - t^4)^{-1}$ if $n:=n_1=n_2$.
  The multiset of weights $\chi$ of $\mathbf{N}= \mathfrak{gl}(2)^{\oplus r} \oplus V$ is
  $\{(1, 0)\}\cup\{(0, 1)\}\cup r \cdot\{(1, -1),\ (-1, 1)\}\cup2r\cdot\{(0,0)\}$, and for
  $\chi = (k_1, k_2)$ and $\lambda = (n_1 \geq n_2)$ we have
  $\langle \chi, \lambda \rangle = n_1k_1 + n_2k_2$. Hence the Hilbert series of the LHS is
  \[(1 -t^2)^{-2} \sum_{n_1 > n_2 \in \mathbb{Z}} t^{(2r -2)(n_1 -n_2) + | n_1 | + |n_2|} + (1 - t^2)^{-1}(1 - t^4)^{-1} \sum_{n \in \mathbb{Z}} t^{2|n|}.\]
  The second sum splits into two summands according to $ n \geq 0$ and $n < 0$ equal to $\frac{1}{(1 - t^2)^2 (1-t^4)}$ and $\frac {t^2}{(1 - t^2)^2 (1-t^4)} $ respectively. The first sum splits into three summands according to $n_1 > n_2 \geq 0$, $0 \geq n_1 > n_2$ and $n_1 > 0 > n_2$ equal to $\frac{t^{2r - 1}}{(1 - t^2)^3 (1 - t^{2r - 1})}$, $\frac{t^{2r - 1}}{(1 - t^2)^3 (1 - t^{2r - 1})}$ and $\frac{t^{2(2r - 1)}}{(1 - t^2)^2 (1 - t^{2r - 1})^2}$ respectively. We get
\begin{multline*}\frac{1}{(1 - t^2)^2 (1-t^4)} + \frac {t^2}{(1 - t^2)^2 (1-t^4)} + \frac{2t^{2r - 1}}{(1 - t^2)^3 (1 - t^{2r - 1})} + \frac{t^{2(2r - 1)}}{(1 - t^2)^2 (1 - t^{2r - 1})^2}=
  \frac{1 - t^{4r}}{(1-t^2)^3(1-t^{2r-1})^2}=H(t).
  \end{multline*}
  \end{proof}

\begin{remark}
  The grading used in the proof of~Lemma~\ref{relatio} gives rise to the following
  contracting $\BC^\times$-action on $\CM_C(r,2,1)$:
  \[c\cdot(x_1, x_2, y_1, y_2, w)=(c^{2r - 1}x_1,\ c^2x_2,\ c^{2r - 1}y_1,\ c^2y_2,\ c^2 w).\]
\end{remark}

\begin{remark}
  Let \[\Omega=\begin{pmatrix} 0 & -1 \\ 1 & 0 \end{pmatrix},\
  N = \begin{pmatrix} -x_2 & \frac{w}{2} \\ \frac{w}{2} & -y_2 \end{pmatrix},\
  A = \begin{pmatrix} x_1 \\ y_1 \end{pmatrix}.\]
  Then~(\ref{starlet}) is equivalent to
  \begin{equation}
    \label{hanany}\tr ((N\Omega)^{2r})=A^{T}\Omega N \Omega A.
  \end{equation}
  up to rescaling $x_1\leadsto 4^rx_1/\sqrt{2},\ y_1\leadsto 4^ry_1/\sqrt{2}.$
  This equation appears for $r=2,3$ in~\cite{Form} for the Slodowy slice
  $\CS(2r-2,1,1)\subset\CN_{\fsp(2r)}$.
  
  \medskip

  We define the action of $SL(2)$ on $\CM_C(r,2,1)$ as follows:
for $S\in SL(2)$ we set
\[S(A)=SA \ (\text{so that}\ A^{T} \mapsto A^{T}S^{T}),\ S(N)=SNS^{T}.\]
Clearly, this action preserves the equation~(\ref{hanany}).
\end{remark}

\subsection{Equations for the Slodowy slice}
\label{slodowy}
All of the matrices in what follows have size $2r \times 2r$ and all the blank spaces are zeroes. 

The matrix 
\[e = \begin{pmatrix} 0 & \ \ 1 \ \ \ & 0 & \cdots & 0 & 0 & 0 \\[6pt] \vdots & \ddots & \ddots & \ddots & \vdots & \vdots & \vdots \\ \vdots & & \ddots & \ddots & 0 & \vdots & \vdots \\ \vdots &  & & \ddots & 2r - 3 & 0 & 0  \\[6pt] 0 & \cdots & \cdots & \cdots & 0 & 0 & 0 \\[6pt] 0 & & \cdots \cdots & \cdots & 0 & 0 & 0 \\[6pt] 0 & \cdots & \cdots & \cdots & 0 & 0 & 0 \end{pmatrix}\]
(right above the main diagonal we have $(1,2,\ldots,2r-4,2r-3,0,0)$)
is of Jordan type $(2r -2, 1, 1)$. Taking
\[f = \begin{pmatrix} 0 & \cdots & \cdots & 0 & 0 & 0 & 0 \\[6pt] 2r - 3 & \ddots & & \vdots & \vdots & \vdots & \vdots \\ 0 & \ddots & \ddots  & \vdots & \vdots & \vdots & \vdots \\[6pt] \vdots & \ddots & \ddots & 0 & \vdots & \vdots & \vdots \\[6pt] 0 & \cdots & 0 & \ \ 1 \ \ \ & 0 & 0 & 0 \\[6pt] 0 &\cdots & \cdots & 0 & 0 & 0 & 0 \\[6pt] 0 &\cdots & \cdots & 0 & 0 & 0 & 0 \end{pmatrix}, \ 
h = \begin{pmatrix} 2r - 3 & 0 & \cdots  & \cdots  & \cdots & 0 & 0 & 0 \\ 0 & \ddots & \ddots & & & \vdots & \vdots & \vdots \\ \vdots & \ddots & 1 & \ddots & & \vdots & \vdots \\[6pt] \vdots & & \ddots & -1 & \ddots & \vdots & \vdots & \vdots \\[6pt] \vdots & & & \ddots & \ddots & 0 & \vdots & \vdots \\[6pt] 0 & \cdots & \cdots & \cdots & 0 & 3-2r & 0 & 0 \\[6pt] 0 & \cdots & \cdots & \cdots & \cdots & 0 & 0 & 0 \\[6pt] 0 & \cdots & \cdots & \cdots & \cdots & 0 & 0 & 0 \end{pmatrix},\]
(right below the main diagonal of $f$ we have $(2r-3,2r-4,\ldots,2,1,0,0)$, and at the main
diagonal of $h$ we have $(2r-3,2r-5,\ldots,5-2r,3-2r,0,0)$)
we see that $\{ e, f, h \}$ is an $\mathfrak{sl}_2$-triple in $\mathfrak{sp}(2r)$, where
$\fsp(2r)$ is the Lie algebra formed by all the matrices preserving the skew-symmetric bilinear
form with the matrix
\[\Omega = \begin{pmatrix} 0 & \cdots & \cdots & \cdots & \cdots & \cdots & 0 & -\binom{2r-3}{0}^{-1} & 0 & 0 \\[6pt] \vdots & & & & &  \iddots & \ \ \binom{2r-3}{1}^{-1} & 0 & \vdots & \vdots \\[6pt] \vdots & & & &  \iddots & \iddots & \iddots & \vdots & \vdots & \vdots \\[6pt] \vdots & & &  \iddots & \pm \binom{2r-3}{r - 2}^{-1} & \iddots & & \vdots & \vdots & \vdots \\[6pt] \vdots & &   \iddots & \mp \binom{2r-3}{r - 1}^{-1} & \iddots & & & \vdots & \vdots & \vdots \\[6pt] \vdots & & \iddots &  \iddots & & & & \vdots & \vdots & \vdots \\[6pt] 0 & -\binom{2r-3}{2r - 2}^{-1} & \iddots & & & & & \vdots & \vdots & \vdots \\[6pt] \ \ \binom{2r-3}{2r-3}^{-1} & 0 & \cdots & \cdots & \cdots & \cdots & \cdots & 0 & 0 & 0 \\[6pt] 0 & \cdots & \cdots & \cdots & \cdots & \cdots & \cdots & 0 & 0 & -1 \\[6pt] 0 & \cdots & \cdots & \cdots & \cdots & \cdots & \cdots & 0 & 1 & 0 \end{pmatrix}\]
(where the coefficients in front of $\binom{2r-3}{r - 2}^{-1}$ and $\binom{2r-3}{r - 1}^{-1}$ are $(-1)^{r-1}$ and $(-1)^r$ respectively). 

The centralizer $Z_{\fsp(2r)}f$ is the set of matrices $A'$ in $\mathfrak{sp}(2r)$ such that
\[\begin{cases} \Omega A' + A'^T \Omega = 0 \\ fA' - A'f = 0. \end{cases}\]
Hence
\[\BS(2r-2,1,1):=e+Z_{\fsp(2r)}f=\left\{A(x_1,y_1,x_2,y_2,w,b_1,\ldots,b_{r-1})\right\}\cong\BA^{r+4},\]
where 
\begin{equation}
  \label{Amatrix}
  A =\begin{pmatrix}
\\ 0  \ \ \                 & 1                   & \ \ 0 \ \ \      & \cdots & \cdots & \cdots & \cdots &   \ \ \ \ 0 \ \ \ \ \ & 0 & 0  \\[6pt]
\binom{2r -3}{1}b_1 & \ddots               &  2      & \ddots &        &        &        & \vdots  & \vdots & \vdots \\[6pt]
\ \ 0 \ \ \                 & \binom{2r - 4}{1}b_1 & \ddots & \ddots & \ddots &        &        & \vdots & \vdots & \vdots \\[6pt]
\binom{2r -3}{3}b_2 & \ddots               & \ddots & \ddots & \ddots & \ddots &        & \vdots & \vdots & \vdots \\[6pt]
\vdots                 & \binom{2r -4}{3}b_2   & \ddots & \ddots & \ddots & \ddots & \ddots & \vdots  & \vdots & \vdots \\[6pt]
\vdots              &              & \ddots & \ddots & \ddots & \ddots & \ddots & \ \ \ \ 0 \ \ \ \ \ & \vdots & \vdots \\[6pt]
\vdots               &                &      & \ddots & \ddots & \ddots & \ddots & 2r -3 & 0 & 0 \\[6pt]
\binom{2r - 3}{2r - 3}b_{r - 1} & \cdots      & \cdots & \cdots    & \ \ \binom{3}{3}b_{2} \ \ \ & \ \ 0 \ \ \ & \ \ \binom{1}{1}b_{1} \ \ \ & \ \ \ 0 \ \ \ \ & cy_1 & cx_1 \\[6pt]
-cx_1 & 0 & \cdots & \cdots & \cdots & \cdots & \cdots & 0 & w & -2x_2 \\[6pt]
 cy_1 & 0 & \cdots & \cdots & \cdots & \cdots & \cdots & 0 & 2y_2 & -w
  \end{pmatrix},
\end{equation}
and $c=\frac{1}{\sqrt{2\cdot(2r-3)!}}$.
To get the equations of the Slodowy slice $\mathcal{S}(2r-2,1,1)$ we need to intersect
$\BS(2r-2,1,1)$ with the nilpotent cone $\CN_{\fsp(2r)}$,
that is to impose the relations that the traces of all
the powers of $A \in\BS(2r-2,1,1)$ are zero.
Computing $\tr A^2, \tr A^4, \cdots, \tr A^{2r-2}$ gives recursively
$b_i = \alpha_i (w^2-4x_2y_2)^i, i = 1, 2,\ldots,r-1$ for some nonzero constants $\alpha_i$. So we have 5 generators $(x_1,y_1,x_2,y_2,w)$ and to get the only relation we calculate the determinant of $A$. 
Consider the general formula for the determinant
$\det A = \sum_{\sigma} (-1)^\sigma a_{1, \sigma(1)} \cdots a_{2r, \sigma(2r)}$
(where $a_{i, j}$ are the matrix elements of $A$) and note that for a permutation $\sigma$ such that $\{ 2r-1, 2r \}$ is not invariant under $\sigma$ to define a nonzero term in the sum one needs to have $\sigma(i) = i + 1$ for $1 \leq i \leq 2r-3$. It follows that
\begin{equation}
  \label{detA}
  \det A=(x_1^2y_2+x_2y_1^2+wx_1y_1)-(w^2-4x_2y_2) \det B,
\end{equation}
where
\begin{equation}
  \label{Bmatrix}
  B =\begin{pmatrix}
\\ 0  \ \ \                 & 1                   & \ \ 0 \ \ \      & \cdots & \cdots & \cdots & \cdots &   \ \ \ \ 0 \ \ \ \ \   \\[6pt]
\binom{2r -3}{1}b_1 & \ddots               &  2      & \ddots &        &        &        & \vdots \\[6pt]
\ \ 0 \ \ \                 & \binom{2r - 4}{1}b_1 & \ddots & \ddots & \ddots &        &        & \vdots \\[6pt]
\binom{2r -3}{3}b_2 & \ddots               & \ddots & \ddots & \ddots & \ddots &        & \vdots \\[6pt]
\vdots                 & \binom{2r -4}{3}b_2   & \ddots & \ddots & \ddots & \ddots & \ddots & \vdots \\[6pt]
\vdots              &              & \ddots & \ddots & \ddots & \ddots & \ddots & \ \ \ \ 0 \ \ \ \ \ \\[6pt]
\vdots               &                &      & \ddots & \ddots & \ddots & \ddots & 2r -3 \\[6pt]
\binom{2r - 3}{2r - 3}b_{r - 1} & \cdots      & \cdots & \cdots    & \ \ \binom{3}{3}b_{2} \ \ \ & \ \ 0 \ \ \ & \ \ \binom{1}{1}b_{1} \ \ \ & \ \ \ 0 \ \ \ \
  \end{pmatrix}
  \end{equation}
is the matrix formed by the first $2r-2$ rows and $2r-2$ columns of $A$.

To compute the determinant of $B$ denote $w^2-4x_2y_2$ by $D$ and note that $B$ is of the form 

\[B =\begin{pmatrix}
\\ 0  \ \ \                 & 1                   & \ \ 0 \ \ \      & \cdots & \cdots & \cdots & \cdots &   \ \ \ \ 0 \ \ \ \ \   \\[6pt]
\beta_{2, 1} D & \ddots               &  2      & \ddots &        &        &        & \vdots \\[6pt]
\ \ 0 \ \ \                 & \beta_{3, 2} D & \ddots & \ddots & \ddots &        &        & \vdots \\[6pt]
\beta_{4, 1} D^2 & \ddots               & \ddots & \ddots & \ddots & \ddots &        & \vdots \\[6pt]
\vdots                 & \beta_{5, 2}D^2   & \ddots & \ddots & \ddots & \ddots & \ddots & \vdots \\[6pt]
\vdots              &              & \ddots & \ddots & \ddots & \ddots & \ddots & \ \ \ \ 0 \ \ \ \ \ \\[6pt]
\vdots               &                &      & \ddots & \ddots & \ddots & \ddots & 2r -3 \\[6pt]
\beta_{2r - 2, 1}D^{r - 1} & \cdots      & \cdots & \cdots    & \ \ \beta_{2r-2, 2r - 5} D^2 \ \ \ & \ \ 0 \ \ \ & \ \ \beta_{2r-2, 2r - 3} D \ \ \ & \ \ \ 0 \ \ \ \
\end{pmatrix}\] for nonzero $\beta_{i, j}$ determined by $\alpha_i$. Then it is easy to show
(e.g.\ by using the row formula repeatedly) that
\begin{equation}
  \label{detB}
  \det B=D^{r-1}=(w^2-4x_2y_2)^{r-1}.
\end{equation}
So we have
\[\det A=(x_1^2y_2+x_2y_1^2+wx_1y_1)-(w^2-4x_2y_2)^r=0.\]
Comparing this equation with~(\ref{starlet}) we have constructed the desired isomorphism
$\Phi$ of~Theorem~\ref{isomo}. It remains to check that $\Phi$ respects the natural Poisson
structures.

%\begin{remark}
%  We can avoid the last step of the above calculation, up to a multiplicative scalar.
%  From the invariant theory there is a unique (up to a constant) invariant function of $B$ that
%is a homogeneous polynomial of fixed degree $n$. Since both $\det B$ and $(\tr B^2)^{r-1}$
%are invariant, they must be proportional.
%  But $\tr B^2 = \tr A^2 - \tr \begin{pmatrix} w & -2x_2 \\ 2y_2 & -w \end{pmatrix}^2=
%  -2(w^2-4x_2y_2)$.
%\end{remark}

\subsection{Poisson structures}
\label{coda}
The generic symplectic form on $\CM_C(r,2,1)$ in the rational \'etale coordinates
$u_1,u_2,w_1,w_2$
of~Section~\ref{coulomb} is $\frac{du_1}{u_1} \wedge dw_1 + \frac{du_2}{u_2} \wedge dw_2$.
In the new variables $v_1:=w_1u_1^{-1},\ v_2:=w_2u_2^{-1}$, we can rewrite this as
$du_1 \wedge dv_1 + du_2 \wedge dv_2$. The corresponding Poisson bracket is
\[\{f, g \} = \left( \frac{\partial f}{\partial u_1} \frac{\partial g}{ \partial v_1} - \frac{\partial f}{\partial v_1} \frac{\partial g}{\partial u_1} \right) + \left( \frac{\partial f}{\partial u_2} \frac{\partial g}{ \partial v_2} - \frac{\partial f}{\partial v_2} \frac{\partial g}{\partial u_2} \right).\]
Recall that we have
\[x_2=(-1)^ru_1u_2, y_2=(-1)^rv_1v_2, w=u_1v_1+u_2v_2,\]
\[x_1=(u_1v_1-u_2v_2)^{r-1}(u_1-(-1)^ru_2), y_1=(u_1v_1-u_2v_2)^{r-1}(v_1-(-1)^rv_2).\]
So we can calculate
\begin{multline}
  \label{poifo}
  \{x_2,y_2\}=w,\ \{w,x_2\}=-2x_2,\ \{w,y_2\}=2y_2,\\
  \{y_2,x_1\}=y_1,\ \{y_2,y_1\}=0=\{x_2,x_1\},\
  \{x_2,y_1\}=-x_1,\ \{w,x_1\}=-x_1,\ \{w,y_1\}=y_1,
  \end{multline}
and $\{x_1,y_1\}$ is determined via the Jacobi identity.

In particular, $\se=(-1)^r\sqrt{-1}x_2,\ \ssf=(-1)^r\sqrt{-1}y_2,\ \sh=-w$
form the standard $\mathfrak{sl}_2$-basis of the vector space $\BC x_2\oplus\BC y_2\oplus\BC w$
of degree 2 generators of $\BC[\CM_C(r,2,1)]$. The vector space $\BC x_1\oplus\BC y_1$ of
degree $2r-1$ generators of $\BC[\CM_C(r,2,1)]$ forms an irreducible 2-dimensional
$\mathfrak{sl}_2$-module. So the Poisson action of $\mathfrak{sl}_2$ integrates to a
Hamiltonian action of $SL(2)$.

A straightforward verification shows that the formulas~(\ref{poifo}) hold true for the
Poisson structure on the Slodowy slice $\CS(2r-2,1,1)$, so that the isomorphism $\Phi$ respects
the Poisson structures. And the above Hamiltonian action of $SL(2)$ is nothing but the
action of $SL(2)\subset Z_{Sp(2r)}(e)$. Theorem~\ref{isomo} is proved. \hfill $\Box$

\section{Concluding remarks}

\subsection{Flavor symmetry}
\label{flavor sym}
We turn on the flavor symmetry $T_F=(\BC^\times)^r$ acting on $\bN=\fgl(V)^{\oplus r}\oplus V$ via
\[(c_1,\ldots,c_r)\cdot(\xi_1,\ldots,\xi_r,v)=(c_1\xi_1,\ldots,c_r\xi_r,v).\]
We denote the generators of $H^\bullet_{T_F}(\on{pt})$ by $z_1,\ldots,z_r$, so that
$H_{T_F}(\on{pt})=\BC[z_1,\ldots,z_r]$. We denote $\on{Spec}H^\bullet_{T_F}(\on{pt})=\on{Lie}T_F$ by
$\fh=\BC z_1\oplus\ldots\oplus\BC z_r$. We denote by $\ul\CM{}_C(r,2,1)$ the corresponding
flavor deformation of $\CM_C(r,2,1)$ over $\fh$.

We identify the diagonal Cartan Lie subalgebra of $\fsp(2r)$ with $\fh$ as follows:
\[(z_1,\ldots,z_r)\mapsto\on{diag}(z_1,\ldots,z_{r-1},-z_{r-1},\ldots,-z_1,z_r,-z_r).\]
The Weyl group $\BW\cong\fS_r\wr\{\pm1\}$ of $\fsp(2r)\supset\fh$ acts naturally on $\fh$, and
$\fh/\!\!/\BW\cong\fsp(2r)/\!\!/\!\on{Ad}_{Sp(2r)}$. Recall the full Slodowy slice
$\BS(2r-2,1,1)=e+Z_{\fsp(2r)}f\subset\fsp(2r)$, see~(\ref{Amatrix}). It projects to
$\fsp(2r)/\!\!/\!\on{Ad}_{Sp(2r)}$, and we define $\ul\BS(2r-2,1,1):=\fh\times_{\fh/\BW}\BS(2r-2,1,1)$.

\begin{prop}
  \label{flavor}
  The isomorphism $\Phi\colon \CM_C(r,2,1)\xrightarrow{\sim}\CS(2r-2,1,1)$ of~Theorem~\ref{isomo}
  extends to an isomorphism of the deformations over $\fh$:
  \[\ul{\Phi}\colon\ul\CM{}_C(r,2,1)\xrightarrow{\sim}\ul\BS(2r-2,1,1).\]
\end{prop}

\begin{proof}
  As in~Section~\ref{coulomb}, using the notations of~\cite[(A.7)]{quiver}, we have the
  following regular functions on $\ul\CM{}_C(r,2,1)$ expressed in terms of the rational
  \'etale coordinates $u_1,u_2,w_1,w_2$:
  \[\sx_1=u_1(w_1-w_2)^{-1}\prod_{i=1}^r(w_1-w_2-z_i)+u_2(w_2-w_1)^{-1}\prod_{i=1}^r(w_2-w_1-z_i),\]
  \[\sy_1=u_1^{-1}w_1(w_1-w_2)^{-1}\prod_{i=1}^r(w_1-w_2+z_i)+
  u_2^{-1}w_2(w_2-w_1)^{-1}\prod_{i=1}^r(w_2-w_1+z_i),\]
  \[\sx_2=(-1)^ru_1u_2,\ \sy_2=(-1)^rw_1w_2u_1^{-1}u_2^{-1},\ \sw=w_1+w_2.\]
  Since $\BC[\CM_C(r,2,1)]=\BC[\ul\CM{}_C(r,2,1)]/(z_1,\ldots,z_r)$ is generated by
  $x_1,x_2,y_1,y_2,w$, we conclude by the graded Nakayama Lemma that
  $\BC[\ul\CM{}_C(r,2,1)]$ is generated by $\sx_1,\sx_2,\sy_1,\sy_2,\sw,z_1,\ldots,z_r$.
  There must be exactly one relation, and one can check that
  \begin{equation}
    \label{relati}
    \BC[\ul\CM{}_C(r,2,1)]=\BC[\sx_1,\sx_2,\sy_1,\sy_2,\sw,z_1,\ldots,z_r]/\big(\sx_1^2\sy_2+
    \sx_2\sy_1^2+\sw\sx_1\sy_1-
 \sum_{\substack{m,n\in\BN\\ m+n\ \on{even}}}(-1)^{mn}\sigma_m\sigma_n(\sw^2-4\sx_2\sy_2)^{r-\frac{m+n}{2}}\big),
  \end{equation}
  where $\sigma_n$ is the $n$-th elementary symmetric polynomial in $z_1,\ldots,z_r$
  (in particular, $\sigma_0=1$, and $\sigma_k=0$ for $k>r$). Note that for any $k$, all the
  powers of $z_i$'s in $\sum_{m+n=2k}(-1)^{mn}\sigma_m\sigma_n$ are even.

  Now recall the matrix $B$ introduced in~(\ref{Bmatrix}) and the equality~(\ref{detA}).
  We assume that the spectrum of the matrix $A$ of~(\ref{Amatrix}) is
  $\{\pm z_1,\ldots,\pm z_r\}$, i.e.\ if $A$ is semisimple, then it is conjugate to
  $\on{diag}(z_1,\ldots,z_{r-1},-z_{r-1},\ldots,-z_1,z_r,-z_r)$. Hence~(\ref{detA}) is equivalent to
  \begin{equation}
    \label{detAB}
    (-1)^rz_1^2\cdots z_r^2=\sx_1^2\sy_2+\sx_2\sy_1^2+\sw\sx_1\sy_1-(\sw^2-4\sx_2\sy_2)\det B.
    \end{equation}
  Similarly to~(\ref{detB}), one can check
  \begin{equation}
    \label{detBB}
    \det B=\sum_{\substack{m,n\in\BN\\ 2r>m+n\
      \on{even}}}(-1)^{mn}\sigma_m\sigma_n(\sw^2-4\sx_2\sy_2)^{r-1-\frac{m+n}{2}}.
  \end{equation}
  Comparing~(\ref{detAB}) and~(\ref{detBB}) with~(\ref{relati}) we deduce the proposition.
\end{proof}

\subsection{Increasing $\dim V$}
\label{3d}
Let us now replace $V=\BC^2$ by $V'=\BC^3$ and consider the corresponding Coulomb branch
$\CM_C(r,3,1)$, a 6-dimensional affine conical Poisson variety. In view of~Theorem~\ref{isomo}
and judging by transversal singularities, it is tempting to hope that $\CM_C(r,3,1)$ might be
isomorphic to the Slodowy slice $\CS(2r-4,2,2)\subset\CN_{\fsp(2r)}$ for $r\geq3$.
We will show that this expectation is false by comparing the Hilbert series of $\CM_C(r,3,1)$
and $\CS(2r-4,2,2)$ already in the case $r=3$.

The Hilbert series of $\CM_C(r,3,1)$ is computed by the monopole formula just as in the
proof of~Lemma~\ref{relatio}.
The dominant coweights of $GL(3)$ are the triples of integers
$\lambda = (n_1 \geq n_2 \geq n_3)$. We have $P_G(t; \lambda) = (1 - t^2)^{-3}$ if
$n_1 > n_2 > n_3;\ P_G(t; \lambda) = (1 - t^2)^{-2}(1-t^4)^{-1}$ if $n_1=n_2>n_3$ or $n_1>n_2=n_3$,
and $P_G(t; \lambda) = (1 - t^2)^{-1}(1 - t^4)^{-1}(1-t^6)^{-1}$ if $n : = n_1 = n_2 = n_3$.
The multiset of weights $\chi$ of $\mathbf{N}=\mathfrak{gl}(3)^{\oplus r} \oplus V'$ is
$\{(1, 0, 0)\}\cup\{(0, 1, 0)\}\cup\{(0, 0, 1)\}\cup r\cdot\{ (1, -1, 0),\ (-1, 1, 0),\
(1, 0, -1),\ (-1, 0, 1),\ (0, 1, -1),\ (0, -1, 1)\}\cup3r\cdot\{(0, 0, 0)\}$.
For $\chi = (k_1, k_2, k_3)$ and $\lambda = (n_1 \geq n_2 \geq n_3)$ we have
$\langle \chi, \lambda \rangle = n_1k_1 + n_2k_2 + n_3k_3$. Hence
\begin{multline*}H'(t)=(1-t^2)^{-3}\sum_{n_1>n_2>n_3\in\mathbb{Z}}t^{(4r-4)(n_1-n_3)+|n_1|+|n_2|+|n_3|}\\
+(1-t^2)^{-2}(1-t^4)^{-1} \left( \sum_{n_1 > n_3 \in \mathbb{Z}} t^{(4r-4)(n_1-n_3) + 2|n_1| + |n_3|} + \sum_{n_1 > n_3 \in \mathbb{Z}} t^{(4r-4)(n_1-n_3) + |n_1| + 2|n_3|} \right)\\
+ (1-t^2)^{-1}(1-t^4)^{-1}(1-t^6)^{-1} \sum_{n \in \mathbb{Z}} t^{3 |n|}.
\end{multline*}
Now $\sum_{n \in \mathbb{Z}} t^{3 |n|}$ splits into two summands according to $ n \geq 0$ and $n < 0$
equal to $\frac{1}{1-t^3}$ and $\frac {t^3}{1-t^3} $ respectively. Furthermore,
$\sum_{n_1 > n_3 \in \mathbb{Z}} t^{(4r-4)(n_1-n_3) + |n_1| + 2|n_3|}$ splits into 3 summands according to
$n_1>n_3 \geq 0$, $0 \geq n_1 > n_3$ and $n_1 > 0 > n_3$ equal to
$\frac{t^{4r-3}}{(1-t^3)(1-t^{4r-3})}$, $\frac{t^{4r-2}}{(1-t^3)(1-t^{4r-2})}$ and
$\frac{t^{4r-3}t^{4r-2}}{(1-t^{4r-3})(1-t^{4r-2})}$ respectively.
Also, $\sum_{n_1 > n_3 \in \mathbb{Z}} t^{(4r-4)(n_1-n_3) + 2|n_1| + |n_3|}$ splits into 3 summands
according to $n_1>n_3 \geq 0$, $0 \geq n_1 > n_3$ and $n_1 > 0 > n_3$ equal to
$\frac{t^{4r-2}}{(1-t^3)(1-t^{4r-2})}$, $\frac{t^{4r-3}}{(1-t^3)(1-t^{4r-3})}$ and
$\frac{t^{4r-3}t^{4r-2}}{(1-t^{4r-3})(1-t^{4r-2})}$ respectively. Finally,
$\sum_{n_1>n_2 > n_3 \in \mathbb{Z}} t^{(4r-4)(n_1-n_3) + |n_1|+|n_2|+|n_3|}$ splits into 5 summands
according to $n_1 > n_2 > n_3 \geq 0$, $0 \geq n_1 > n_2 > n_3$, $n_1 > n_2 > 0> n_3$,
$n_1 > 0 > n_2 > n_3$ and $n_1 > (n_2 = 0) > n_3$ equal to
$\frac{t^{4r-3}t^{4r-2}}{(1-t^3)(1-t^{4r-3})(1-t^{4r-2})}$,
$\frac{t^{4r-3}t^{4r-2}}{(1-t^3)(1-t^{4r-3})(1-t^{4r-2})}$,
$\frac{t^{2(4r-3)}t^{4r-2}}{(1-t^{4r-3})^2 (1-t^{4r-2})}$,
$\frac{t^{2(4r-3)}t^{4r-2}}{(1-t^{4r-3})^2 (1-t^{4r-2})}$ and $\frac{t^{2(4r-3)}}{(1-t^{4r-3})^2}$
respectively.

Summing up we have \[H'(t) =\frac{(1-t^{4r})(1+ t^{4r-2} + 2t^{4r-1} + t^{4r} + t^{8r-2})}{(1-t^2)(1-t^3)^2 (1-t^4) (1-t^{4r-3})^2 (1-t^{4r-2})}.\]
We conclude that $\CM_C(r,3,1)$ is not a complete intersection. This is well known for the
Jordan quiver with $r=1$, when we have $\CM_C(1,3,1)\simeq\on{Sym}^3(\BA^2)$
(see e.g.~\cite[Proposition~3.24]{quiver}).

On the other hand, the Hilbert series of the Slodowy slices $\CS(2r-4,2,2)$ were computed
in~\cite[Tables~18,19]{Form} for $r=3,4$. The answers are

\[\frac{(1-t^8)(1-t^{12})}{(1-t^2)^3 (1-t^4)^5},\
\frac{(1-t^{12})(1-t^{16})}{(1-t^2) (1-t^4)^5 (1-t^6)^2}\] respectively.

One can similarly show that the Coulomb branch $\CM_C(r,2,m)$ for $m > 1$ is not a complete
intersection.

\end{document}